\documentclass[12pt]{amsart}
\usepackage{amsmath}
\usepackage{amsthm}
\usepackage{amssymb}
\usepackage{amsfonts}
\usepackage{bm}
\usepackage[shortlabels]{enumitem}
\usepackage[bookmarks=true,hyperindex,pdftex,colorlinks,citecolor=red,linkcolor=blue]{hyperref}
\usepackage{centernot}
\usepackage{marginnote}
\usepackage{bbm}
\usepackage[all]{xy}
\usepackage{tikz}
\usepackage{comment}
\usetikzlibrary{arrows}
\usetikzlibrary{shapes,decorations}
\usetikzlibrary{positioning}
\usepackage{xcolor}
\usepackage{bigints}
\usepackage{cancel}
\usepackage[
a4paper,
left=2.4cm,
right=2.4cm,
top=2.3cm,
bottom=2.5cm,
headheight=14pt,
includeheadfoot
]{geometry}

\theoremstyle{plain}
\newtheorem{thm}{Theorem}[section]
\newtheorem{prop}[thm]{Proposition}
\newtheorem{lemma}[thm]{Lemma}
\newtheorem{cor}[thm]{Corollary}

\theoremstyle{remark}
\newtheorem{remark}[thm]{Remark}
\newtheorem{example}[thm]{Example}

\theoremstyle{definition}

\newtheorem{q}[thm]{Question}

\newcommand{\N}{\mathbb{N}}
\newcommand{\R}{\mathbb{R}}
\newcommand{\Z}{\mathbb{Z}}

\newcommand{\eps}{\varepsilon}

\definecolor{darkgreen}{rgb}{.2,.6,.2}
\title{Self-improvement for absolutely Ces\`aro bounded operators}
\author[L. Arnold]{Loris Arnold}

\address[L. Arnold]{}
\email{lfj.arld@gmail.com}
\begin{document}
	
	\begin{abstract}
		We prove that every $p$-absolutely Ces\`aro bounded operator with
		constant $K>1$ is $q$-absolutely Ces\`aro bounded for
		$1\leq q<pK/(K-1)$. This range is optimal, even for positive
		weighted shifts. On Banach spaces of type $s\in (1,2]$, we prove that every absolutely
		Ces\`aro bounded operator is $q$-absolutely Ces\`aro bounded
		for every $1\leq q<s$, with a logarithmic estimate at $q=s$.
		We also prove the polynomial growth bound
		$\|T^n\|=O(n^{1/s-\eps})$ for some $\eps>0$.
		On $L^p$-spaces, $1<p<\infty$, we extend the implication
		from absolute Ces\`aro boundedness to $p$-absolute Ces\`aro
		boundedness from positive to individually eventually positive
		operators. Finally, for every $1\leq p<\infty$, we construct
		a $p$-absolutely Ces\`aro bounded operator $T$ on $\ell^1$
		such that $\sigma(T)\subseteq\mathbb D\cup\{1\}$ but
		$\limsup_{n\to\infty}\|T^n(I-T)\|=\infty$. We also observe that the construction yields a single operator
		which is $p$-absolutely Ces\`aro bounded for every $1\leq p<\infty$.
	\end{abstract}
	
	\maketitle

	\section{Introduction}
	
	Let $1\leq p<\infty$. A bounded operator $T$ on a Banach space $X$
	is called \emph{$p$-absolutely Ces\`aro bounded}, or $p$-ACB,
	if there exists $C>0$ such that, for every $x\in X$ and every
	$N\geq1$,
	\[
	\frac1N\sum_{k=0}^{N-1}\|T^kx\|^p\leq C\|x\|^p.
	\]
	We denote the least admissible constant by $C_{p,\mathrm{ac}}(T)$.
	When $p=1$, we simply say that $T$ is absolutely Ces\`aro bounded
	and write $C_{\mathrm{ac}}(T)=C_{1,\mathrm{ac}}(T)$.
	The argument $T$ is omitted when there is no ambiguity.
	By Jensen's inequality, $p$-ACB implies $q$-ACB for $1\leq q\leq p$.
	
	Absolute Ces\`aro boundedness was introduced by Luo and Hou
	\cite{HouLuo} and studied by Berm\'udez, Bonilla,
	M\"uller and Peris \cite{BermudezBonillaMullerPeris}.
	They constructed absolutely Ces\`aro bounded operators that are
	not power bounded and proved that
	$\|T^n\|=o(n)$ on arbitrary Banach spaces and
	$\|T^n\|=o(n^{1/2})$ on Hilbert spaces.
	Abadias and Bonilla \cite{AbadiasBonilla} extended the study to
	absolute Ces\`aro boundedness of fractional order and obtained
	growth and ergodic results. The case $p=2$, known as Ces\`aro square boundedness,
	had previously been studied by van Casteren
	\cite{VanCasteren1983}.
	Cohen, Cuny, Eisner and Lin introduced the general $p$-ACB
	condition and proved that it implies
	\[
	\|T^n\|=O(n^{1/p-\eps})
	\]
	for some $\eps>0$, see
	\cite[Proposition~3.1]{CohenCunyEisnerLin}.
	They also showed that every absolutely Ces\`aro bounded operator
	on a Hilbert space is $2$-ACB, and hence satisfies
	$\|T^n\|=O(n^{1/2-\eps})$ for some $\eps>0$;
	see \cite[Theorem~4.4]{CohenCunyEisnerLin}.
	
	Cuny \cite{CunyLp} related these estimates to the type and cotype
	of the underlying space. More precisely, if $X$ has type
	$s\in (1,2]$ with finite cotype $r \in [2,\infty)$, and $T$ is absolutely Ces\`aro
	bounded, then
	\[
	\|T^n\|
	=O\left(\frac{n^{1/s}}{(\log n)^{1/r}}\right).
	\]
	Our polynomial improvement of this estimate relies on a randomized
	orbit estimate of Cuny, which we recall in the preliminary section.
	
	For $1<p<\infty$, $L^p$-spaces have type $\min\{p,2\}$
	and cotype $\max\{p,2\}$.
	The implication from absolute Ces\`aro boundedness to $p$-ACB
	is therefore a natural extension of the Hilbert space result. For $1<p<2$, whether this implication holds was explicitly
	asked by Cuny \cite[Section~2]{CunyLp}.
	For $p>2$, this implication fails: for every $q>2$,
	Cuny's examples in \cite[Proposition~2.8]{CunyLp} provide
	$2$-absolutely Ces\`aro bounded operators on $L^p(\mathbb T)$
	which are not $q$-absolutely Ces\`aro bounded.
	Under positivity, however, Abbar, Arnold and Coine
	\cite[Theorem~2.1]{AbbarArnoldCoine} proved that absolute
	Ces\`aro boundedness implies $p$-ACB on $L^p(\Omega)$
	for every $1<p<\infty$.
	They also characterized $q$-absolute Ces\`aro boundedness
	of backward shifts on weighted $\ell^p$-spaces for $q\geq p$;
	see \cite[Theorem~3.4]{AbbarArnoldCoine}.
	
	Our first result is an optimal self-improvement theorem on
	arbitrary Banach spaces. If $T$ is $p$-ACB and
	$K=C_{p,\mathrm{ac}}(T)>1$, then $T$ is $q$-ACB for every
	\[
	1\leq q<\frac{pK}{K-1}.
	\]
	If $K=1$, then $T$ is a contraction.
	The upper endpoint cannot be increased or included, even for
	positive weighted shifts on $\ell^p$.
	In particular, every absolutely Ces\`aro bounded operator
	is $q$-ACB for some $q>1$.
	The proof combines an equivalent renorming with a distribution
	estimate for orbit norms; see
	Theorem~\ref{thmACBSelfImprovement}.
	
	We next obtain a weak endpoint estimate determined by type.
	If $X$ has type $s\in(1,2]$ and $T$ is absolutely
	Ces\`aro bounded, then, for every $x\in X$, $N\geq1$ and
	$\lambda>0$,
	\[
	\#\{0\leq k<N:\|T^kx\|>\lambda\}
	\leq CN\frac{\|x\|^s}{\lambda^s}.
	\]
	Consequently, $T$ is $q$-ACB for every $1\leq q<s$, and
	\[
	\frac1N\sum_{k=0}^{N-1}\|T^kx\|^s
	\leq C\log(N+1)\|x\|^s.
	\]
	These results are proved in Section~\ref{secWeakEndpoint}.
	In particular, they give all exponents $q<p$ for absolutely
	Ces\`aro bounded operators on $L^p(\Omega)$ when $1<p<2$.
	
	Section~\ref{secPolynomialGap} improves Cuny's logarithmic
	gain to a polynomial one. If $X$ has type $s\in(1,2]$ and $T$ is absolutely
	Ces\`aro bounded, we prove that
	\[
	\|T^n\|=O(n^{1/s-\eps})
	\]
	for some $\eps>0$.
	The argument combines the randomized orbit estimate of Cuny recalled
	in Proposition~\ref{prop:Cuny-randomized} with a finite-dimensional
	weighted shift associated with each orbit segment.
	For $L^p(\Omega)$, this gives
	$\|T^n\|=O(n^{1/\min\{p,2\}-\eps})$.
	
	In Section~\ref{secEventualPositivity}, we extend the positive
	endpoint theorem of Abbar, Arnold and Coine to
	\emph{individually eventually positive} operators:
	for each $f\geq0$, one only assumes that $T^nf\geq0$
	for all sufficiently large $n$, with the threshold depending
	on $f$.
	Under this assumption, absolute Ces\`aro boundedness on
	$L^p(\Omega)$ implies $p$-ACB for every $1<p<\infty$.
	The proof uses an eventual domination argument based on
	\cite[Theorem~2.1]{ArnoldIEP}.
	Together with \cite[Proposition~3.1]{CohenCunyEisnerLin},
	this yields $\|T^n\|=O(n^{1/p-\eps})$.
	
	Finally, we consider differences of consecutive powers.
	The Katznelson--Tzafriri theorem states that a power-bounded
	operator satisfying $\sigma(T)\cap\mathbb T\subseteq\{1\}$
	also satisfies $\|T^n(I-T)\|\to0$ as $n\to\infty$,
	see \cite{KatznelsonTzafriri}.
	In Proposition~\ref{prop:ACB-counterexample-KT}, we construct,
	for every $1\leq p<\infty$, a $p$-ACB operator on $\ell^1$
	such that
	\[
	C_{p,\mathrm{ac}}(T)\leq2,
	\qquad
	\sigma(T)\subseteq\mathbb D\cup\{1\},
	\qquad
	\limsup_{n\to\infty}\|T^n(I-T)\|=\infty.
	\]
	Taking $p=2$ answers negatively
	\cite[Section~6, Question~7]{CohenCunyEisnerLin}, which asks
	whether Ces\`aro square boundedness suffices.
	The uniform boundedness principle also rules out strong
	convergence, answering the accompanying question. 
	The construction uses the fact that, for $\delta\in[0,1]$,
	$(1-\delta)I+\delta T$ preserves $p$-absolute Ces\`aro
	boundedness without increasing its constant. A modification yields a single operator which is
	$p$-absolutely Ces\`aro bounded for every $1\leq p<\infty$
	and still satisfies the same spectral condition and divergence.

	We conclude with questions concerning the $L^p$ endpoint
	for $1<p<2$ and the Katznelson--Tzafriri conclusion under
	positivity.
	
	\subsection*{Notation and conventions}
	
	We write $\N_0=\{0,1,2,\ldots\}$,
	$\mathbb D=\{z\in\mathbb C:|z|<1\}$ and
	$\mathbb T=\{z\in\mathbb C:|z|=1\}$.

	Unless otherwise stated, $X$ is a real or complex Banach space
	and $T\in\mathcal L(X)$.
	Statements involving the spectrum are understood over the
	complex field. We write $p'=p/(p-1)$ for $1<p<\infty$.
	In Section~\ref{secEventualPositivity}, $L^p(\Omega)$ denotes
	the complex $L^p$-space over a measure space
	$(\Omega,\mathcal F,\mu)$.
	Except in Section~\ref{secWeakEndpoint}, where two independent
	randomizations are used, $\mathbb E$ denotes expectation
	with respect to the Rademacher variables under consideration.
	
	\section{Preliminaries}
	\label{secPreliminaries}
	
	\subsection{Rademacher estimates}
	
	Let $(\varepsilon_k)$ be a sequence of independent Rademacher variables. For
	$r\geq1$, we denote by $K_r$ an admissible Kahane--Khintchine constant such
	that
	\[
	\left(\mathbb E\left\|\sum_{k=1}^m\varepsilon_kx_k\right\|^r\right)^{1/r}
	\leq K_r\,\mathbb E\left\|\sum_{k=1}^m\varepsilon_kx_k\right\|
	\]
	for every Banach space $X$, every $m\geq1$ and every
	$x_1,\ldots,x_m\in X$. We shall only use this form of the
	Kahane--Khintchine inequalities; see \cite[Theorem~2.1]{CunyLp}.
	
	We also record precisely the two forms of Kahane's contraction principle which
	will be used below; see \cite[Proposition~2.2]{CunyLp}, which in turn follows
	from \cite[Theorem~6.1.13(ii)]{HytNeerVeraarWeisII}. We recall that an $X$-valued random variable is called $\mathbb R$-symmetric
	if it has the same distribution as its negative.
	
	\begin{prop}[Kahane's contraction principle]\label{propContraction}
		Let $\xi_1,\ldots,\xi_m$ be independent, integrable, and $\R$-symmetric
		$X$-valued random variables. If $\varnothing\neq I\subseteq J\subseteq\{1,\ldots,m\}$, then
		\begin{equation}\label{eq:contraction-subset}
			\mathbb E\left\|\sum_{k\in I}\xi_k\right\|
			\leq
			\mathbb E\left\|\sum_{k\in J}\xi_k\right\|.
		\end{equation}
		Moreover, if $(b_k)_{k\in I}$ is a family of nonzero scalars, then
		\begin{equation}\label{eq:contraction-weights}
			\mathbb E\left\|\sum_{k\in I}\xi_k\right\|
			\leq
			\frac{\pi}{2}\max_{k\in I}|b_k|^{-1}
			\mathbb E\left\|\sum_{k\in I}b_k\xi_k\right\|.
		\end{equation}
	\end{prop}
	
	Thus, in the form most frequently used below, terms may be removed from a
	Rademacher sum with constant $1$.
	
	\subsection{Type and cotype}
	
	Let $s\in[1,2]$. We say that $X$ has \emph{type $s$} if there exists
	$\tau_{s,X}>0$ such that
	\begin{equation}
		\label{eq:type}
		\left(\mathbb E\left\|\sum_{k=1}^m\varepsilon_kx_k\right\|^s\right)^{1/s}
		\leq \tau_{s,X}\left(\sum_{k=1}^m\|x_k\|^s\right)^{1/s}
	\end{equation}
	for every $m\geq1$ and every $x_1,\ldots,x_m\in X$. Likewise, for
	$r\in[2,\infty)$, $X$ has \emph{cotype $r$} if there exists $\gamma_{r,X}>0$ such
	that
	\[
	\left(\sum_{k=1}^m\|x_k\|^r\right)^{1/r}
	\leq \gamma_{r,X}
	\left(\mathbb E\left\|\sum_{k=1}^m\varepsilon_kx_k\right\|^r\right)^{1/r}
	\]
	for every $m\geq1$ and every $x_1,\ldots,x_m\in X$. We refer to
	\cite[Chapter~7]{HytNeerVeraarWeisII}. In particular, $L^p(\Omega)$ has type
	$\min\{p,2\}$ and cotype $\max\{p,2\}$ whenever $1<p<\infty$. Recall also that every Banach space of non-trivial type
	has finite cotype; see \cite[Remark~2.4]{CunyLp}.
	
	We shall also use the following standard consequence of finite cotype. If $X$
	has cotype $r<\infty$, there exists $\chi_{r,X}>0$ such that
	\begin{equation}\label{eq:independent-cotype}
		\left(\sum_{k=1}^m\mathbb E\|\zeta_k\|^r\right)^{1/r}
		\leq \chi_{r,X}
		\left(\mathbb E\left\|\sum_{k=1}^m\zeta_k\right\|^r\right)^{1/r}
	\end{equation}
	for every finite family of independent centered $X$-valued random variables
	$\zeta_1,\ldots,\zeta_m$ with finite $r$-th moments. This is
	\cite[Proposition~9.11]{LedouxTalagrand}; see also the discussion preceding
	\cite[Theorem~2.3]{CunyLp}.
	
	We conclude this section with an important estimate of Cuny
	\cite[Theorem~2.3, estimate~(9)]{CunyLp}, which will be a key
	ingredient in the sequel.
	
	\begin{prop}\label{prop:Cuny-randomized}
		Let $X$ be a Banach space of type $s\in (1,2]$, and let $T\in\mathcal L(X)$
		be absolutely Ces\`aro bounded. Then there exists a constant $C>0$
		such that
		\[
		\left(
		\mathbb E\left\|\sum_{k=0}^{N-1}\varepsilon_kT^kx\right\|^s
		\right)^{1/s}
		\leq C N^{1/s}\|x\|
		\]
		for every $x\in X$ and every $N\geq1$, where
		$(\varepsilon_k)_{k\geq0}$ is a sequence of independent Rademacher
		variables.
	\end{prop}
	\section{Self-improvement beyond the initial exponent}
	
	We first establish a self-improvement property which holds on arbitrary
	Banach spaces. The range of exponents depends on the absolute
	Ces\`aro bound and is optimal, even for positive operators on
	$\ell^p$.
	
	\begin{thm}\label{thmACBSelfImprovement}
		Let $X$ be a nonzero Banach space, let $1\leq p<\infty$, and let
		$T\in\mathcal L(X)$ be $p$-absolutely Ces\`aro bounded.
		If $C_{p,\mathrm{ac}}=1$, then $T$ is a contraction.
		If $C_{p,\mathrm{ac}}>1$, then $T$ is $q$-absolutely Ces\`aro bounded for every
		\[
		1\leq q<\frac{pC_{p,\mathrm{ac}}}{C_{p,\mathrm{ac}}-1}.
		\]
	\end{thm}
	
	\begin{proof}
		If $C_{p,\mathrm{ac}}=1$, the $p$-absolute Ces\`aro estimate with $N=2$ gives
		\[
		\|x\|^p+\|Tx\|^p\leq2\|x\|^p
		\qquad(x\in X),
		\]
		so that $\|T\|\leq1$.
		
		Assume now that $C_{p,\mathrm{ac}}>1$. By Jensen's inequality, it suffices to consider
		\[
		p<q<\frac{pC_{p,\mathrm{ac}}}{C_{p,\mathrm{ac}}-1}.
		\]
		Set $r=q/p>1$. Then
		\begin{equation*}\label{eq:ACB-exponent-condition}
			C_{p,\mathrm{ac}}(r-1)<r.
		\end{equation*}
		
		We first construct an equivalent norm in which the
		$p$-absolute Ces\`aro bound remains at most $C_{p,\mathrm{ac}}$
		and the operator norm of $T$ is at most $1+\eta$,
		where $\eta>0$ is arbitrary.
		For $m\geq1$, define
		\[
		\|x\|_m
		=
		\left(\sum_{j=0}^{m-1}\|T^jx\|^p\right)^{1/p}.
		\]
		This is an equivalent norm on $X$, with
		\begin{equation}\label{eq:ACB-equivalent-norm}
			\|x\|\leq\|x\|_m\leq(C_{p,\mathrm{ac}}m)^{1/p}\|x\|.
		\end{equation}
		Moreover, for every $x\in X$ and $N\geq1$,
		\begin{align}
			\sum_{k=0}^{N-1}\|T^kx\|_m^p
			&=\sum_{j=0}^{m-1}\sum_{k=0}^{N-1}\|T^kT^jx\|^p
			\notag\\
			&\leq C_{p,\mathrm{ac}}N\sum_{j=0}^{m-1}\|T^jx\|^p
			=C_{p,\mathrm{ac}}N\|x\|_m^p.
			\label{eq:ACB-renormed-bound}
		\end{align}

		For every $0\leq j<m$, we have
		\[
		\|T^mx\|^p
		\leq \|T^{m-j}\|^p\|T^jx\|^p.
		\]
		Summing over $j$ gives
		\[
		\|T^mx\|^p
		\leq \frac1m\max_{1\leq k\leq m}\|T^k\|^p
		\|x\|_m^p.
		\]
		Consequently,
		\begin{align*}
			\|Tx\|_m^p
			&=\|x\|_m^p-\|x\|^p+\|T^mx\|^p\\
			&\leq D_m\|x\|_m^p,
		\end{align*}
		where
		\[
		D_m:=1+\frac1m\max_{1\leq k\leq m}\|T^k\|^p
		\xrightarrow[m\to\infty]{}1.
		\]
		Indeed, \cite[Proposition~3.1]{CohenCunyEisnerLin} gives
		$\|T^k\|^p=O((k+1)^{1-\delta})$ for some $\delta\in(0,1)$.
		
		Since $D_m\xrightarrow[m\to\infty]{}1$, we may fix $m$ sufficiently large that
		\begin{equation}\label{eq:ACB-absorption-condition}
			C_{p,\mathrm{ac}}D_m(r-1)<r.
		\end{equation}
		
		Fix $x\neq0$ and $N\geq1$. By
		\eqref{eq:ACB-renormed-bound} and the preceding operator norm
		estimate, for every $j\geq0$ and every $\ell\geq1$,
		\begin{equation}\label{eq:ACB-orbit-sequence}
			\sum_{k=j}^{j+\ell-1}\|T^kx\|_m^p
			\leq C_{p,\mathrm{ac}}\ell\|T^jx\|_m^p,
			\qquad
			\|T^{j+1}x\|_m^p\leq D_m\|T^jx\|_m^p.
		\end{equation}
		
		For $\lambda\geq0$, define
		\[
		E_\lambda
		=\{0\leq j<N:\|T^jx\|_m^p>\lambda\}.
		\]
		Suppose that $\lambda\geq\|x\|_m^p$, and write $E_\lambda$ as a disjoint union of maximal intervals.
		If $I=\{a,\ldots,b\}$ is one such interval, then
		$a\geq1$ and $\|T^{a-1}x\|_m^p\leq\lambda$.
		Hence $\|T^ax\|_m^p\leq D_m\lambda$, and
		\eqref{eq:ACB-orbit-sequence} gives
		\[
		\sum_{j\in I}\|T^jx\|_m^p
		\leq C_{p,\mathrm{ac}}|I|\|T^ax\|_m^p
		\leq C_{p,\mathrm{ac}}D_m\lambda|I|.
		\]
		Summing over these intervals yields, for every
		$\lambda\geq\|x\|_m^p$,
		\begin{equation}\label{eq:ACB-distribution-bound}
			\sum_{j\in E_\lambda}\|T^jx\|_m^p
			\leq C_{p,\mathrm{ac}}D_m\lambda\,\#E_\lambda.
		\end{equation}
		
		Using the identity
		\[
		a^r=(r-1)\int_0^a\lambda^{r-2}a\,d\lambda,
		\qquad a\geq0,
		\]
		followed by \eqref{eq:ACB-distribution-bound} and the
		layer-cake formula, we obtain
		\begin{align*}
			\sum_{j=0}^{N-1}\|T^jx\|_m^{pr}
			&=(r-1)\int_0^\infty
			\lambda^{r-2}
			\sum_{j\in E_\lambda}\|T^jx\|_m^p\,d\lambda\\
			&\leq
			\|x\|_m^{p(r-1)}
			\sum_{j=0}^{N-1}\|T^jx\|_m^p
			+C_{p,\mathrm{ac}}D_m(r-1)\int_{\|x\|_m^p}^\infty
			\lambda^{r-1}\#E_\lambda\,d\lambda\\
			&\leq
			C_{p,\mathrm{ac}}N\|x\|_m^{pr}
			+\frac{C_{p,\mathrm{ac}}D_m(r-1)}{r}
			\sum_{j=0}^{N-1}\|T^jx\|_m^{pr}.
		\end{align*}
		By \eqref{eq:ACB-absorption-condition}, this implies
		\[
		\frac1N\sum_{j=0}^{N-1}\|T^jx\|_m^q
		\leq
		\frac{C_{p,\mathrm{ac}}}{1-C_{p,\mathrm{ac}}D_m(r-1)/r}\,\|x\|_m^q.
		\]
		The case $x=0$ is immediate. Finally,
		\eqref{eq:ACB-equivalent-norm} yields
		\[
		\frac1N\sum_{j=0}^{N-1}\|T^jx\|^q
		\leq
		\frac{C_{p,\mathrm{ac}}(C_{p,\mathrm{ac}}m)^r}{1-C_{p,\mathrm{ac}}D_m(r-1)/r}\,\|x\|^q.
		\]
		Thus $T$ is $q$-absolutely Ces\`aro bounded.
	\end{proof}
	
	The upper endpoint in Theorem~\ref{thmACBSelfImprovement} cannot
	be increased or included, even for positive weighted shifts.
	
	\begin{example}\label{exSharpACBExponent}
		Fix $1\leq p<\infty$ and $K>1$, and set
		\[
		\beta=1-\frac1K,
		\qquad
		\alpha=\frac{\beta}{p},
		\qquad
		q_0=\frac{pK}{K-1}.
		\]
		On $\ell^p(\mathbb N)$, define the positive weighted backward
		shift $W$ by $We_1=0$ and, for every $j\geq2$,
		\[
		We_j=\left(\frac{j}{j-1}\right)^\alpha e_{j-1}.
		\]
		Then $C_{p,\mathrm{ac}}(W)=K$, and $W$ is $q$-absolutely
		Ces\`aro bounded if and only if $1\leq q<q_0$. Indeed, by \cite[Theorem~3.4(3) and the proof of Corollary~3.5]
		{AbbarArnoldCoine},
		\[
		C_{p,\mathrm{ac}}(W)
		=\sup_{j\geq1}j^{\beta-1}\sum_{\ell=1}^j\ell^{-\beta}
		=\frac1{1-\beta}=K,
		\]
		where the second equality follows from the integral comparisons \[
		\frac{(j+1)^{1-\beta}-1}{1-\beta}
		\leq \sum_{\ell=1}^j\ell^{-\beta}
		\leq \frac{j^{1-\beta}}{1-\beta}.
		\]
		
		Theorem~\ref{thmACBSelfImprovement} therefore gives
		$q$-absolute Ces\`aro boundedness for every $1\leq q<q_0$.
		On the other hand, since $\alpha q_0=1$,
		\[
		\frac1j\sum_{k=0}^{j-1}\|W^ke_j\|_p^{q_0}
		=\sum_{\ell=1}^j\frac1\ell
		\longrightarrow\infty.
		\]
		Thus $W$ is not $q_0$-absolutely Ces\`aro bounded.
		By Jensen's inequality, it is not $q$-absolutely Ces\`aro
		bounded for any $q\geq q_0$.
	\end{example}
	
	\begin{remark}\label{remIndividualACBThreshold}
		For a $p$-absolutely Ces\`aro bounded operator $T$, the constant
		$C_{p,\mathrm{ac}}(T)>1$ determines an optimal universal
		lower bound for
		\[
		q_*(T)=\sup\{q\geq p:T\text{ is }q\text{-absolutely Ces\`aro bounded}\},
		\]
		but does not determine $q_*(T)$ itself.
		Indeed, for $K>1$, the positive operator on $\ell^p(\{1,2\})$
		defined by
		\[
		S(x_1,x_2)=\bigl((2K-1)^{1/p}x_2,0\bigr)
		\]
		satisfies $S^2=0$ and $C_{p,\mathrm{ac}}(S)=K$.
		Nevertheless, $S$ is power bounded and hence $q$-absolutely
		Ces\`aro bounded for every $q\geq1$.
		Thus $q_*(S)=\infty$, whereas the operator in
		Example~\ref{exSharpACBExponent} has the same $p$-absolute
		Ces\`aro bound and satisfies $q_*(W)=pK/(K-1)$.
	\end{remark}
	
	\section{Self-improvement under non-trivial type}
	\label{secWeakEndpoint}
	
	IIn this section two independent randomizations will be used.
	For each $\delta\in(0,1]$, let $(\Sigma,\mathcal A,\mathbb P)$
	be the canonical product probability space carrying an i.i.d.
	Rademacher sequence $(\varepsilon_k)_{k\in\Z}$ and an independent
	i.i.d. Bernoulli sequence $(\eta_k)_{k\in\Z}$ with
	$\mathbb P(\eta_k=1)=\delta$.
	We suppress the dependence on $\delta$ in the notation.
	We write $\mathbb{E}_\varepsilon$ and $\mathbb E_\eta$ for the
	corresponding expectations, $\mathbb P_\eta$ for the Bernoulli
	probability, and $\mathbb E=\mathbb E_\eta\mathbb E_\varepsilon$
	for expectation on $(\Sigma,\mathcal A,\mathbb P)$.
	
	For $x\in X$, $N\geq1$ and $\lambda>0$, let
	\[
	d_{x,N}(\lambda)=\#\{0\leq k<N:\ \|T^kx\|>\lambda\}.
	\]
	
	\begin{thm}\label{thmWeakEndpoint}
		Let $X$ be a Banach space of type $s\in(1,2]$. If $T$ is
		absolutely Ces\`aro bounded, then there exists $C>0$ such that
		\[
		d_{x,N}(\lambda)
		\leq C N\frac{\|x\|^s}{\lambda^s}
		\]
		for every $x\in X$, every $N\geq1$ and every $\lambda>0$.
	\end{thm}
	
	\begin{proof}
		The case $x=0$ is immediate. By homogeneity, it is enough
		to consider $x\in X$ with $\|x\|=1$.  We first prove a uniform estimate on dyadic
		blocks. Let $M=2^m$ with $m\geq1$, and set
		$I_M=\{M/2,\ldots,M-1\}$. We claim that
		\begin{equation}\label{eq:dyadic-weak-estimate}
			\#\{n\in I_M:\ \|T^nx\|>\lambda\}\lambda^s
			\leq CM
			\qquad(\lambda>0).
		\end{equation}
		
		Fix $\delta\in(0,1]$ and set $\xi_k:=\varepsilon_k\eta_k$. Let $U\in\mathcal L(L^1(\Sigma))$ be an isometric Koopman operator
		satisfying $U\xi_k=\xi_{k+1}$ for every $k\in\Z$.
		We denote by $\widetilde U\in\mathcal L(L^1(\Sigma;X))$
		its canonical extension to $L^1(\Sigma;X)$ and by
		$\widetilde T\in\mathcal L(L^1(\Sigma;X))$ the pointwise extension
		of $T$, defined by $(\widetilde T F)(\omega)=T(F(\omega))$.
		These operators commute, and for every $j \geq 0$ and every $F\in L^1(\Sigma;X)$,
		\begin{equation}\label{stationarity}
			\mathbb E\|\widetilde U^{\,j}F\|
			=\mathbb E\|F\| .
		\end{equation}
		Define
		\[
		y_M=\sum_{k=1}^M\xi_k\frac{T^kx}{1+\|T^kx\|},
		\qquad
		u_M=\sum_{j=0}^{M-1}\widetilde T^{\,j}\widetilde U^{\,j}y_M.
		\]
		Conditioning on the Bernoulli variables and applying
		\eqref{eq:type}, we obtain
		\begin{align*}
			\mathbb E\|y_M\|^s
			&=\mathbb E_\eta\mathbb E_\varepsilon
			\left\|\sum_{k=1}^M
			\varepsilon_k\eta_k\frac{T^kx}{1+\|T^kx\|}\right\|^s \\
			&\leq \tau_{s,X}^s\mathbb E_\eta
			\sum_{k=1}^M\eta_k\frac{\|T^kx\|^s}{(1+\|T^kx\|)^s}
			\leq \tau_{s,X}^s\delta M.
		\end{align*}
		Hence by Jensen's inequality, $\mathbb E\|y_M\|\leq \tau_{s,X}(\delta M)^{1/s}$. Then, by stationarity \eqref{stationarity} and
		absolute Ces\`aro boundedness,
		\begin{equation}\label{eq:uM-upper-bound}
			\mathbb E\|u_M\|
			\leq C_{\mathrm{ac}}M\,\mathbb E\|y_M\|
			\leq C_{\mathrm{ac}}\tau_{s,X}M(\delta M)^{1/s}.
		\end{equation}
		
		For $1\leq n\leq M$, the coefficient of
		$\varepsilon_n\eta_nT^nx$ in the expansion of $u_M$ is
		\[
		b_n:=\sum_{k=1}^n(1+\|T^kx\|)^{-1}.
		\] Since $\|x\|=1$, absolute Ces\`aro
		boundedness gives
		$\sum_{k=1}^{M/2}(1+\|T^kx\|)\leq(C_{\mathrm{ac}}+1/2)M$. Thus
		Cauchy--Schwarz yields $b_n\geq\kappa M$ for every $n\in I_M$, where
		$\kappa=(4C_{\mathrm{ac}}+2)^{-1}$.
		
		For every realisation of the Bernoulli variables, applying
		\eqref{eq:contraction-subset} and then
		\eqref{eq:contraction-weights}, we obtain
		\begin{align*}
			\mathbb E_\varepsilon\|u_M\|
			&\geq
			\mathbb E_\varepsilon
			\left\|\sum_{n\in I_M}
			\varepsilon_n\eta_n b_nT^nx\right\|\\
			&\geq
			\frac{2\kappa M}{\pi}\,
			\mathbb E_\varepsilon
			\left\|\sum_{n\in I_M}
			\varepsilon_n\eta_nT^nx\right\|,
		\end{align*}
		where we used $b_n\geq\kappa M$ for every $n\in I_M$.
		Moreover, applying \eqref{eq:contraction-subset} with
		$I=\{k\}$ and $J=I_M$ gives, for every $k\in I_M$,
		\[
		\mathbb E_\varepsilon
		\left\|\sum_{n\in I_M}\varepsilon_n\eta_nT^nx\right\|
		\geq
		\mathbb E_\varepsilon\|\varepsilon_k\eta_kT^kx\|
		=\eta_k\|T^kx\|.
		\]
		Taking the maximum over $k\in I_M$, we obtain
		\[
		\mathbb E_\varepsilon
		\left\|\sum_{n\in I_M}\varepsilon_n\eta_nT^nx\right\|
		\geq \max_{k\in I_M}\eta_k\|T^kx\|.
		\]
		Taking expectation with respect to the Bernoulli variables
		and using \eqref{eq:uM-upper-bound}, we obtain
		\begin{equation}\label{eq:Bernoulli-maximum}
			\mathbb E_\eta\max_{n\in I_M}\eta_n \|T^nx\|
			\leq C(\delta M)^{1/s},
		\end{equation}
		where $C$ is independent of $M$, $\delta$ and the unit vector $x$.
		
		Let
		\[
		E_\lambda=\{n\in I_M:\|T^nx\|>\lambda\},
		\qquad R=\#E_\lambda.
		\]
		If $R=0$, there is nothing to prove. Otherwise, choose
		$\delta=R^{-1}$. Independence of the Bernoulli variables yields
		\begin{align*}
			\mathbb E_\eta\max_{n\in I_M}\eta_n\|T^nx\|
			&\geq
			\lambda\,\mathbb P_\eta
			\left(\sum_{n\in E_\lambda}\eta_n\geq1\right)\\
			&=\lambda\left[
			1-\mathbb P_\eta
			\left(\bigcap_{n\in E_\lambda}\{\eta_n=0\}\right)
			\right]\\
			&=\lambda\left[1-\left(1-\frac1R\right)^R\right]\\
			&\geq(1-e^{-1})\lambda.
		\end{align*}
		Combining this with \eqref{eq:Bernoulli-maximum} gives
		\[
		(1-e^{-1})\lambda\leq C(M/R)^{1/s},
		\]
		and hence $R\lambda^s\leq CM$, proving \eqref{eq:dyadic-weak-estimate}.
		
		Let now $N\geq2$ and choose $J\geq1$ such that
		$2^{J-1}<N\leq2^J$. Since
		$\{0,\ldots,2^J-1\}=\{0\}\cup\bigcup_{j=1}^JI_{2^j}$
		and $\|x\|=1$, estimate~\eqref{eq:dyadic-weak-estimate} gives
		\[
		d_{x,N}(\lambda)
		\leq \mathbf 1_{\{\lambda<1\}}
		+C\lambda^{-s}\sum_{j=1}^J2^j
		\leq C'N\lambda^{-s}.
		\]
		The case $N=1$ is immediate, and homogeneity completes the proof.
	\end{proof}
	
	\begin{cor}\label{corSubendpointACB}
		Under the assumptions of Theorem~\ref{thmWeakEndpoint}, $T$ is
		$q$-absolutely Ces\`aro bounded for every $1\leq q<s$.
	\end{cor}
	
	\begin{proof}
		Let $x\in X$ and $N\geq1$. The case $x=0$ is immediate, so assume $x\neq0$.
		Theorem~\ref{thmWeakEndpoint} gives
		$d_{x,N}(\lambda)\leq\min\{N,CN\|x\|^s\lambda^{-s}\}$ for every
		$\lambda>0$. Hence, with $\lambda_0=C^{1/s}\|x\|$, the layer-cake formula
		yields
		\begin{align*}
			\sum_{k=0}^{N-1}\|T^kx\|^q
			&=q\int_0^\infty\lambda^{q-1}d_{x,N}(\lambda)\,d\lambda \\
			&\leq qN\int_0^{\lambda_0}\lambda^{q-1}\,d\lambda
			+qCN\|x\|^s\int_{\lambda_0}^\infty\lambda^{q-s-1}\,d\lambda \\
			&=\frac{s}{s-q}C^{q/s}N\|x\|^q.
		\end{align*}
		The resulting bound depends on $q$ and diverges as $q\uparrow s$.
	\end{proof}
	\begin{remark}\label{remLpWeakEndpoint}
		Since $L^p(\Omega)$ has type $\min\{p,2\}$ for $1<p<\infty$,
		Corollary~\ref{corSubendpointACB} implies that every absolutely
		Ces\`aro bounded operator on $L^p(\Omega)$ is
		$q$-absolutely Ces\`aro bounded for every
		$1\leq q<\min\{p,2\}$.
	\end{remark}
	\begin{cor}\label{corEndpointLogarithm}
		Under the assumptions of Theorem~\ref{thmWeakEndpoint}, there exists $C>0$
		such that
		\[
		\frac1N\sum_{k=0}^{N-1}\|T^kx\|^s
		\leq C\log(N+1)\|x\|^s
		\]
		for every $x\in X$ and every $N\geq1$.
	\end{cor}
	
	\begin{proof}
		The case $x=0$ is immediate, so assume $x\neq0$.
		By Theorem~\ref{thmWeakEndpoint}, there exists $C\geq1$ such that
		\[
		d_{x,N}(\lambda)
		\leq\min\{N,CN\|x\|^s\lambda^{-s}\}
		\qquad(\lambda>0).
		\]
		Since $d_{x,N}(\lambda)$ is a non-negative integer, this also gives
		\[
		d_{x,N}(\lambda)=0
		\qquad\text{whenever }\lambda>(CN)^{1/s}\|x\|.
		\]
		Set $\lambda_0=C^{1/s}\|x\|$. By the layer-cake formula,
		\begin{align*}
			\sum_{k=0}^{N-1}\|T^kx\|^s
			&=s\int_0^\infty
			\lambda^{s-1}d_{x,N}(\lambda)\,d\lambda\\
			&\leq
			sN\int_0^{\lambda_0}\lambda^{s-1}\,d\lambda
			+sCN\|x\|^s
			\int_{\lambda_0}^{N^{1/s}\lambda_0}
			\frac{d\lambda}{\lambda}\\
			&=CN(1+\log N)\|x\|^s\\
			&\leq C'N\log(N+1)\|x\|^s.
		\end{align*}
	\end{proof}

	\section{Power growth under non-trivial type}
	\label{secPolynomialGap}
	
	Let $\mathbb K$ denote the scalar field of $X$. Let $y_0,\ldots,y_n\in X$ satisfy $\|y_j\|=1$ for every $0\leq j\leq n$,
	and let $(\varepsilon_j)_{j=0}^n$ be independent Rademacher variables. On
	$\mathbb K^{n+1}$ define
	\[
	\|c\|_{E_y}
	=\mathbb E\left\|\sum_{j=0}^n\varepsilon_jc_jy_j\right\|.
	\]
	By \eqref{eq:contraction-subset}, for $0 \le k \le n$, $|c_k|\leq\|c\|_{E_y}$. Thus $\|\cdot\|_{E_y}$ is a norm. The canonical basis vectors satisfy $\|e_j\|_{E_y}=1$
	for every $0\leq j\leq n$, 
	and every coordinate projection is contractive. 
	
	\begin{lemma}\label{lemRademacherDisjointCotype}
		Assume that $X$ has finite cotype $r$. Then, for every family $y_0,\ldots,y_n\in X$ satisfying
		$\|y_j\|=1$ for all $0\leq j\leq n$, and every $c^{(1)},\ldots,c^{(m)}\in\mathbb K^{n+1}$ with
		pairwise disjoint supports,
		\[
		\left(\sum_{i=1}^m\|c^{(i)}\|_{E_y}^r\right)^{1/r}
		\leq \theta_{r,X}
		\left\|\sum_{i=1}^mc^{(i)}\right\|_{E_y},
		\]
		where $\theta_{r,X}=K_r\chi_{r,X}$.
	\end{lemma}
	
	\begin{proof}
		For $1\leq i\leq m$, let
		$Y_i=\sum_{j\in\operatorname{supp}c^{(i)}}
		\varepsilon_jc_j^{(i)}y_j$. The variables $Y_1,\ldots,Y_m$ are independent
		and centred. By H\"older's inequality, \eqref{eq:independent-cotype} and the
		Kahane--Khintchine inequality,
		\begin{align*}
			\left(\sum_{i=1}^m\|c^{(i)}\|_{E_y}^r\right)^{1/r}
			&\leq\left(\sum_{i=1}^m\mathbb E\|Y_i\|^r\right)^{1/r} \\
			&\leq \chi_{r,X}
			\left(\mathbb E\left\|\sum_{i=1}^mY_i\right\|^r\right)^{1/r} \\
			&\leq \chi_{r,X}K_r\,
			\mathbb E\left\|\sum_{i=1}^mY_i\right\|,
		\end{align*}
		which is the desired estimate.
	\end{proof}
	
	\begin{thm}\label{thmPolynomialGap}
		Let $X$ have type $s\in(1,2]$ and denote by $r r\in[2,\infty)$ its cotype. If $T$ is
		absolutely Ces\`aro bounded, then there exist $C>0$ and $\eps>0$ such that
		\[
		\|T^n\|\leq C(n+1)^{1/s-\eps}
		\qquad (n\geq0).
		\]
	\end{thm}
	
	\begin{proof}
		Fix $x\in X$ and $n\geq2$, and set $a_j=\|T^jx\|$ for $0\leq j\leq n$.
		We may assume that $\|T^nx\|>0$; then $a_j>0$ for every $0\leq j\leq n$. Put
		$y_j=T^{n-j}x/a_{n-j}$ and equip $\mathbb{K}^{n+1}$ with the norm defined above.
		
		Let $B$ be the backward weighted shift on $E_y$ given by $Be_0=0$ and
		$Be_j=\beta_je_{j-1}$ for $1\leq j\leq n$, where
		$\beta_j=a_{n-j+1}/a_{n-j}$. We first observe that $B$ is absolutely Ces\`aro
		bounded with constant $C_{\mathrm{ac}}$. Indeed, for $0\leq k\leq j\leq n$, we have
		\[
		B^ke_j=\frac{a_{n-j+k}}{a_{n-j}}e_{j-k},
		\qquad
		T^ky_j=\frac{a_{n-j+k}}{a_{n-j}}y_{j-k}.
		\]
		Thus, for every $c=(c_0,\ldots,c_n)\in\mathbb K^{n+1}$
		and $0\leq k\leq n$,
		\begin{align*}
			\|B^kc\|_{E_y}
			&=\mathbb E\left\|
			\sum_{j=k}^n\varepsilon_{j-k}c_j
			\frac{a_{n-j+k}}{a_{n-j}}y_{j-k}
			\right\|\\
			&=\mathbb E\left\|
			\sum_{j=k}^n\varepsilon_jc_jT^ky_j
			\right\|\\
			&\leq\mathbb E\left\|
			T^k\sum_{j=0}^n\varepsilon_jc_jy_j
			\right\|.
		\end{align*}
		The second equality follows because
		$(\varepsilon_{j-k})_{j=k}^n$ and $(\varepsilon_j)_{j=k}^n$
		have the same distribution, and the inequality follows from
		\eqref{eq:contraction-subset}.
		For $k>n$, the same inequality holds since $B^k=0$.
		Consequently, absolute Ces\`aro boundedness of $T$ gives,
		for every $N\geq1$,
		\[
		\sum_{k=0}^{N-1}\|B^kc\|_{E_y}
		\leq
		\mathbb E\sum_{k=0}^{N-1}
		\left\|T^k\sum_{j=0}^n\varepsilon_jc_jy_j\right\|
		\leq C_{\mathrm{ac}}N\|c\|_{E_y}.
		\]

		Set $\Pi_0=1$ and
		$\Pi_j=\beta_1\cdots\beta_j=\|T^nx\|/a_{n-j}$ for $1\leq j\leq n$. Choose
		$L\geq0$ such that $3^{L+1}\leq n+1<3^{L+2}$. For $0\leq\ell\leq L$, let
		$m_\ell=3^\ell$ and define
		\[
		z_\ell=\sum_{j=m_\ell}^{3m_\ell-1}\Pi_j^{-1}e_j,
		\qquad
		v_\ell=\sum_{j=0}^{m_\ell-1}\Pi_j^{-1}e_j.
		\]
		These vectors are well defined by the choice of $L$. If
		$m_\ell\leq k<2m_\ell$, then
		$B^k(\Pi_{k+j}^{-1}e_{k+j})=\Pi_j^{-1}e_j$ for every
		$0\leq j<m_\ell$. Hence the projection of $B^kz_\ell$ onto
		$\operatorname{span}\{e_0,\ldots,e_{m_\ell-1}\}$ is $v_\ell$. Since
		coordinate projections are contractive, summing over these values of $k$ gives, for $0\le \ell \le L$,
		
		\begin{align*}
			m_\ell\|v_\ell\|_{E_y}
			\leq
			\sum_{k=m_\ell}^{2m_\ell-1}\|B^kz_\ell\|_{E_y}\leq
			\sum_{k=0}^{2m_\ell-1}\|B^kz_\ell\|_{E_y}\leq
			2C_{\mathrm{ac}}m_\ell\|z_\ell\|_{E_y}.
		\end{align*}
		
		Let $Z_\ell=\|z_\ell\|_{E_y}$ and put
		$\alpha=2C_{\mathrm{ac}}\theta_{r,X}$.
		For $1\leq\ell\leq L$, the identity
		$v_\ell=e_0+\sum_{j=0}^{\ell-1}z_j$ and
		Lemma~\ref{lemRademacherDisjointCotype} give
		\begin{equation}\label{ineSL}
			\left(1+\sum_{j=0}^{\ell-1}Z_j^r\right)^{1/r}
			\leq \theta_{r,X}\|v_\ell\|_{E_y}
			\leq \alpha Z_\ell.
		\end{equation}
		The same inequality holds for $\ell=0$, with the convention
		that an empty sum is zero, since $v_0=e_0$ and
		$\theta_{r,X}\geq1$.
		
		Set $S_{-1}=1$ and $S_\ell=1+\sum_{j=0}^\ell Z_j^r$.
		FFor every $0\leq\ell\leq L$, \eqref{ineSL} gives
		$Z_\ell^r\geq\alpha^{-r}S_{\ell-1}$, and hence
		\[
		S_\ell=S_{\ell-1}+Z_\ell^r
		\geq(1+\alpha^{-r})S_{\ell-1}.
		\]
		Induction therefore gives
		$S_{\ell-1}\geq(1+\alpha^{-r})^\ell$ for $0\leq\ell\leq L$.
		Consequently, setting $\rho=(1+\alpha^{-r})^{1/r}>1$, we obtain, for $0\leq\ell\leq L$,
		\begin{equation*}\label{estiZL}
			Z_\ell\geq\alpha^{-1}\rho^\ell.
		\end{equation*}
		
		On the other hand, $\Pi_j^{-1}y_j=T^{n-j}x/a_n$. Therefore
		\[
		a_nZ_L
		=\mathbb E\left\|
		\sum_{j=3^L}^{3^{L+1}-1}\varepsilon_jT^{n-j}x
		\right\|.
		\]
		Reindexing the Rademacher variables and applying
		\eqref{eq:contraction-subset}, Jensen's inequality, and Proposition~\ref{prop:Cuny-randomized}, we obtain
		\begin{align*}
			a_nZ_L
			\leq \mathbb E\left\|
			\sum_{k=0}^n\varepsilon_kT^kx\right\|\leq
			\left(\mathbb E\left\|
			\sum_{k=0}^n\varepsilon_kT^kx\right\|^s\right)^{1/s}\leq C(n+1)^{1/s}\|x\|.
		\end{align*} By \eqref{estiZL}, it follows that
		\[
		\|T^nx\|
		=a_n
		\leq \frac{C(n+1)^{1/s}\|x\|}{Z_L}
		\leq C\alpha(n+1)^{1/s}\rho^{-L}\|x\|.
		\]
		Since $L>\log_3(n+1)-2$ and $\rho>1$, we have
		\[
		\rho^{-L}
		\leq \rho^2(n+1)^{-\log_3\rho}.
		\]
		Absorbing $\alpha\rho^2$ into the constant, we obtain
		\[
		\|T^nx\|\leq C(n+1)^{1/s-\eps}\|x\|,
		\]
		where $\eps=\log_3\rho
		=\frac1r\log_3(1+\alpha^{-r})>0$ only on $C_{\mathrm{ac}}$,
		$r$, and the cotype constants of $X$. 
		Taking the supremum over $\|x\|=1$ gives the desired estimate
		for $n\geq2$. Increasing $C$ if necessary also covers $n=0,1$.
		
	\end{proof}
	
	\begin{remark}\label{remLpPolynomialGap}
		Since $L^p(\Omega)$ has type $\min\{p,2\}$ and cotype
		$\max\{p,2\}$ for $1<p<\infty$, Theorem~\ref{thmPolynomialGap}
		implies that every absolutely Ces\`aro bounded operator $T$
		on $L^p(\Omega)$ satisfies, for every $n \ge 0$,
		\[
		\|T^n\|\leq C(n+1)^{1/\min\{p,2\}-\eps}
		\]
		for some $C>0$ and $\eps>0$.
	\end{remark}
	
	\section{Absolute Cesàro boundedness under individual eventual positivity}
	\label{secEventualPositivity}
	
	Let $E$ be a complex Banach lattice. We write $E_{\R}$ for its real part. Recall that
	$T\in\mathcal L(E)$ is \emph{individually eventually positive} if,
	for every $f\in E_+$, there exists $n_f\in\N_0$ such that
	$T^nf\geq0$ for every $n\geq n_f$. Eventual positivity has been studied for powers of operators \cite{GlueckEventualOperators}, although most of the literature concerns $C_0$-semigroups on Banach lattices, see, for instance, \cite{DanersGlueckKennedy}.
	For $g\in E_+$, we denote by
	\[
	E_g=\{h\in E:|h|\leq cg\text{ for some }c\geq0\}
	\]
	the principal ideal generated by $g$, equipped with the gauge norm
	\[
	\|h\|_g=\inf\{c\geq0:|h|\leq cg\}.
	\]
	Then $E_g$ is a Banach space, and its inclusion into $E$ is continuous.
	
	We shall use the following discrete form of
	\cite[Corollary~2.2]{ArnoldIEP}. 
	
	\begin{lemma}\label{lemEventualDomination}
		Let $T\in\mathcal L(E)$ be individually eventually positive, let $g\in E_+$,
		let $Y$ be a real Banach space, and let $J:Y\to E_{\R}$ be a bounded operator with
		$JY\subseteq E_g$. Then there exist $n_0\in\N_0$ and $C>0$ such that
		\[
		|T^nJy|\leq C\|y\|_Y T^ng
		\]
		for every $y\in Y$ and every $n\geq n_0$. Moreover, $T^ng\geq0$ for every
		$n\geq n_0$.
	\end{lemma}
	\begin{proof}
		By the closed graph theorem, $J:Y\to E_g$ is bounded. Hence there exists $M>0$ such that
		$|Jy|\leq M\|y\|_Yg$ for every $y\in Y$.
		Choose $n_g\in\N_0$ such that $T^ng\geq0$ for all $n\geq n_g$.
		Individual eventual positivity, applied to
		$M\|y\|_Yg\pm Jy$, yields
		\[
		|T^nJy|\leq M\|y\|_Y T^ng
		\]
		for all sufficiently large $n$, depending on $y$.
		The conclusion follows from \cite[Theorem~2.1]{ArnoldIEP},
		applied to $S_n=T^{n}J$ and $W_n=E_{T^{n}g}$,  $n\geq n_g$.
	\end{proof}
	\begin{thm}\label{thmEventualPositiveEndpoint}
		Let $(\Omega,\mathcal F,\mu)$ be a measure space, let $1<p<\infty$, and let
		$T\in\mathcal L(L^p(\Omega))$ be individually eventually positive. If $T$ is
		absolutely Ces\`aro bounded, then $T$ is $p$-absolutely Ces\`aro bounded.
	\end{thm}
	
	\begin{proof}
		By dyadic decomposition and the uniform boundedness principle,
		it suffices to prove that, for every $f\in L^p(\Omega)_+$ with
		$\|f\|_p=1$,
		\[
		\sup_{m\geq0}B_m<\infty,
		\qquad
		B_m:=2^{-m}\sum_{n=2^m}^{2^{m+1}-1}\|T^nf\|_p^p,
		\]
		since every element of $L^p(\Omega)$ is a linear combination
		of positive functions.
		
		Fix such an $f$, and choose $n_f\in\N_0$ such that
		$T^nf\geq0$ for every $n\geq n_f$.
		If $T^kf=0$ for some $k$, the desired bound is immediate.
		We may therefore assume that $\|T^kf\|_p>0$ for every $k\geq0$.
		
		Choose $m_0\geq0$ such that $2^{m_0}\geq n_f+1$, and let
		$b=(b_m)_{m\geq m_0}\in\ell^p$ be arbitrary. Put
		$I_m=\{2^m,\ldots,2^{m+1}-1\}$ and define
		\[
		w_k:=
		\begin{cases}
			|b_m|2^{-m/p}, & \text{if }k\in I_m\text{ for some }m\geq m_0,\\
			0, & \text{otherwise}.
		\end{cases}
		\] Then
		$\sum_{k\geq0}w_k^p=\sum_{m\geq m_0}|b_m|^p<\infty$.
		Set
		\[
		v_k:=
		\begin{cases}
			\displaystyle\frac{T^kf}{\|T^kf\|_p}, & \text{if }w_k\neq0,\\
			0, & \text{otherwise}.
		\end{cases}
		\] Then $v_k\geq0$
		for every $k\geq0$, and $\|v_k\|_p=1$ whenever $w_k\neq0$. Hence
		\[
		g:=\left(\sum_{k=0}^\infty w_k^pv_k^p\right)^{1/p}
		\]
		belongs to $L^p(\Omega)_+$ and satisfies
		$\|g\|_p^p=\sum_{k\geq0}w_k^p$.
		
		Define the real-linear operator
		\[
		J:\ell^{p'}(\N_0;\R)\longrightarrow L^p(\Omega;\R),
		\qquad
		J\alpha:=\sum_{k=0}^\infty\alpha_kw_kv_k,
		\]
		where the series converges absolutely in $L^p(\Omega)$, since
		H\"older's inequality gives
		\[
		\sum_{k=0}^\infty\|\alpha_kw_kv_k\|_p
		\leq\sum_{k=0}^\infty|\alpha_k|w_k
		\leq\|\alpha\|_{\ell^{p'}}\|w\|_{\ell^p}.
		\]Moreover, H\"older's inequality gives
		\[
		|J\alpha|
		\leq \sum_{k=0}^\infty|\alpha_k|w_kv_k
		\leq \|\alpha\|_{\ell^{p'}}
		\left(\sum_{k=0}^\infty w_k^pv_k^p\right)^{1/p}
		=\|\alpha\|_{\ell^{p'}}g,
		\]
		where the inequalities are understood in the Banach lattice $L^p(\Omega)$. Hence
		$J:\ell^{p'}(\N_0;\R)\to L^p(\Omega;\R)$ is bounded and its range is
		contained in $E_g$. Lemma~\ref{lemEventualDomination} yields
		$n_0\in\N_0$ and $C>0$ such that
		\[
		|T^jJ\alpha|\leq C\|\alpha\|_{\ell^{p'}}T^jg
		\]
		for every $\alpha\in\ell^{p'}(\N_0;\R)$ and every $j\geq n_0$.
		If $w_k\neq0$, then $k\geq n_f$ and hence
		$T^jv_k=T^{j+k}f/\|T^kf\|_p\geq0$ for every $j\geq n_0$. The $\ell^p$--$\ell^{p'}$ duality therefore gives for $j\ge n_0$
		\[
		\left(\sum_{k=0}^\infty w_k^pT^jv_k^p\right)^{1/p}
		\leq CT^jg.
		\]

		For $j\geq n_0$, let $F_j\in L^p(\Omega;\ell^p(\N_0))$ be given by \[
		F_j:=(w_kT^{j+k}f/\|T^kf\|_p)_{k\geq0}\] and, let $R$ be the right shift on
		$\ell^p(\N_0)$. By the preceding domination estimate,
		\[
		\|F_j\|_{\ell^p}
		:=
		\left(
		\sum_{k=0}^\infty
		\frac{w_kT^{j+k}f}{\|T^kf\|_p}^p
		\right)^{1/p}
		=
		\left(\sum_{k=0}^\infty w_k^pT^jv_k^p\right)^{1/p}
		\leq CT^jg,
		\]
		where the inequality is understood in the Banach lattice
		$L^p(\Omega)$. Hence, for every
		$L\geq1$, since $R$ is an isometry on $\ell^p(\N_0)$, Minkowski's inequality gives
		\[
		\left\|\sum_{j=n_0}^{n_0+L-1}R^jF_j\right\|_{\ell^p}
		\leq \sum_{j=n_0}^{n_0+L-1}\|R^jF_j\|_{\ell^p}
		= \sum_{j=n_0}^{n_0+L-1}\|F_j\|_{\ell^p}
		\leq C\sum_{j=n_0}^{n_0+L-1}T^jg.
		\]
		For every $n\geq0$, the $n$-th coordinate of the sum is
		\[
		\left(\sum_{j=n_0}^{n_0+L-1}R^jF_j\right)_n
		=
		\sum_{\substack{n_0\leq j\leq n_0+L-1\\j\leq n}}
		\frac{w_{n-j}T^nf}{\|T^{n-j}f\|_p}
		=T^nf\,H_{n,L},
		\]
		where, setting $k=n-j$,
		\[
		H_{n,L}
		:=
		\sum_{\substack{k\geq0\\n_0\leq n-k\leq n_0+L-1}}
		\frac{w_k}{\|T^kf\|_p}.
		\]
		Taking $L^p$-norms and applying Fubini's theorem, we obtain
		\begin{align*}
			\sum_{n=0}^\infty \|T^nf\|_p^pH_{n,L}^p
			&=\int_\Omega\sum_{n=0}^\infty
			|T^nf|^pH_{n,L}^p\,d\mu\\
			&=\left\|
			\left\|\sum_{j=n_0}^{n_0+L-1}R^jF_j\right\|_{\ell^p}
			\right\|_p^p\\
			&\leq C^p
			\left\|\sum_{j=n_0}^{n_0+L-1}T^jg\right\|_p^p.
		\end{align*}
		Moreover, absolute Ces\`aro boundedness gives
		$\sum_{j=n_0}^{n_0+L-1}\|T^jg\|_p
		\leq C_{\mathrm{ac}}L\|T^{n_0}g\|_p$. Consequently, there exists $C_b>0$,
		independent of $L$, such that for $L \ge 1 $
		\begin{equation}\label{ineTnHnL}
			\sum_{n=0}^\infty \|T^nf\|_p^pH_{n,L}^p\leq C_b L^p.
		\end{equation}
		
		Fix $m\geq m_0$ sufficiently large that $M:=2^m\geq2(n_0+1)$,
		and take $L=3M$.
		If $2M\leq n\leq4M-1$ and $M\leq k\leq2M-n_0-1$, then
		\[
		n_0+1\leq n-k\leq3M-1\leq n_0+3M-1.
		\]
		Thus every such $k$ occurs in the sum defining $H_{n,3M}$.
		Moreover, $k\in I_m$, so that $w_k=|b_m|M^{-1/p}$.
		Consequently,
		\[
		H_{n,3M}
		\geq |b_m|M^{-1/p}
		\sum_{k=M}^{2M-n_0-1}\frac{1}{\|T^kf\|_p}.
		\]
		There are $M-n_0$ terms in this sum. By Cauchy--Schwarz,
		\[
		(M-n_0)^2
		\leq
		\left(\sum_{k=M}^{2M-n_0-1}\|T^kf\|_p\right)
		\left(\sum_{k=M}^{2M-n_0-1}\frac{1}{\|T^kf\|_p}\right).
		\]
		Hence
		\[
		H_{n,3M}
		\geq |b_m|M^{-1/p}
		\frac{(M-n_0)^2}
		{\displaystyle\sum_{k=M}^{2M-n_0-1}\|T^kf\|_p}.
		\]
		Since $\sum_{k=M}^{2M-n_0-1}\|T^kf\|_p\leq2C_{\mathrm{ac}}M$ and $M-n_0\geq M/2$, we obtain
		\[
		H_{n,3M}\geq \frac{|b_m|M^{1-1/p}}{8C_{\mathrm{ac}}}.
		\]
		Restricting the sum in the estimate \eqref{ineTnHnL} to $2M\leq n\leq4M-1$, we obtain
		$|b_m|^pB_{m+1}\leq C C_b$ for all sufficiently large $m$. Thus
		\[
		\sup_{m\geq m_0}|b_m|^pB_{m+1}<\infty
		\]
		for every $b\in\ell^p$. The uniform boundedness principle, applied to the
		coordinate functionals
		$b\mapsto B_{m+1}^{1/p}b_m$, now gives
		$\sup_{m\geq m_0}B_{m+1}<\infty$. Since only finitely many indices remain,
		$\sup_{m\geq0}B_m<\infty$.
		
		The reduction at the beginning of the proof now gives the conclusion.
	\end{proof}
	
	\begin{remark}
		For positive operators, Theorem~\ref{thmEventualPositiveEndpoint}
		recovers \cite[Theorem~2.1]{AbbarArnoldCoine}: a positive operator
		on $L^p(\Omega)$, $1<p<\infty$, is absolutely Ces\`aro bounded
		if and only if it is $p$-absolutely Ces\`aro bounded.
		The reverse implication follows from Jensen's inequality.
	\end{remark}

	\begin{cor}\label{corEventualPositivePolynomialGrowth}
		Let $1<p<\infty$ and let $T$ be individually eventually positive and absolutely
		Ces\`aro bounded on $L^p(\Omega)$. Then there exist $C>0$ and
		$\eps\in(0,1/p)$ such that
		\[
		\|T^n\|\leq C(n+1)^{1/p-\eps}
		\qquad (n\geq0).
		\]
	\end{cor}
	
	\begin{proof}
		By Theorem~\ref{thmEventualPositiveEndpoint}, $T$ is $p$-absolutely Ces\`aro
		bounded. The conclusion follows from
		\cite[Proposition~3.1]{CohenCunyEisnerLin}.
	\end{proof}

	\section{Failure of the Katznelson--Tzafriri conclusion under absolute Cesàro boundedness}
	
	In this section, we give a negative answer to
	\cite[Section~6, Question~7]{CohenCunyEisnerLin}
	and to the case $\alpha=1$ of
	\cite[Question~3.1]{AbadiasBonilla}
	by constructing, for every $1\leq p<\infty$, a
	$p$-absolutely Ces\`aro bounded operator $T$ on $\ell^1$
	such that $\sigma(T)\subseteq\mathbb D\cup\{1\}$ but
	$\limsup_{n\to\infty}\|T^n(I-T)\|=\infty$.
	\begin{prop}\label{prop:ACB-counterexample-KT}
		For every $1\leq p<\infty$, there exists a
		$p$-absolutely Ces\`aro bounded operator $T$ on $\ell^1$ such that
		\[
		\sigma(T)\subseteq\mathbb D\cup\{1\},
		\qquad
		\limsup_{n\to\infty}\|T^n(I-T)\|=\infty.
		\]
		Moreover, $T$ can be chosen so that $C_{p,\mathrm{ac}}(T)\leq2$.
	\end{prop}
	\begin{remark}
		By the uniform boundedness principle, there exists $x\in X$
		such that
		\[
		\sup_{n\geq0}\|T^n(I-T)x\|=\infty.
		\]
		In particular, $T^n(I-T)$ does not converge strongly to zero.
	\end{remark}
	We start with a lemma which is of independent interest. 
	\begin{lemma}\label{lem:ACB-convex-identity}
		Let $1\leq p<\infty$ and let $T$ be $p$-absolutely Ces\`aro bounded
		on a Banach space $X$.
		For every $\delta\in[0,1]$, the operator
		\[
		S_\delta=(1-\delta)I+\delta T
		\]
		is $p$-absolutely Ces\`aro bounded, with
		\[
		C_{p,\mathrm{ac}}(S_\delta)
		\leq C_{p,\mathrm{ac}}(T).
		\]
	\end{lemma}
	
	\begin{proof}
		The cases $\delta=0,1$ are immediate. Fix $0<\delta<1$,
		$N\geq1$ and $x\in X$. The binomial formula and Jensen's
		inequality give
		\[
		\frac1N\sum_{n=0}^{N-1}\|S_\delta^nx\|^p
		\leq\sum_{k=0}^{N-1}a_k\|T^kx\|^p,
		\]
		where
		\[
		a_k=\frac1N\sum_{n=k}^{N-1}
		\binom nk\delta^k(1-\delta)^{n-k}.
		\]
		Setting for fixed $k$, $F_k = 0$ and
		$F_n :=\sum_{\ell=k+1}^{n}\binom n\ell
		\delta^\ell(1-\delta)^{n-\ell}$ for $n>k$. Pascal's identity yields an expression for $F_{n+1}-F_n$. Summing it from $n=k$ to $N-1$ gives
		\[
		\sum_{n=k}^{N-1}\binom nk
		\delta^{k+1}(1-\delta)^{n-k}
		=\sum_{\ell=k+1}^{N}\binom N\ell
		\delta^\ell(1-\delta)^{N-\ell}.
		\]
		This shows that $(a_k)_{k=0}^{N-1}$ is non-increasing.
		Moreover,
		\[
		\sum_{k=0}^{N-1}a_k
		=\frac1N\sum_{n=0}^{N-1}
		\sum_{k=0}^{n}\binom nk\delta^k(1-\delta)^{n-k}
		=1.
		\]
		Setting $a_N=0$ and applying Abel's summation formula, we obtain
		\begin{align*}
			\sum_{k=0}^{N-1}a_k\|T^kx\|^p
			&=\sum_{j=0}^{N-1}(a_j-a_{j+1})
			\sum_{k=0}^{j}\|T^kx\|^p\\
			&\leq C_{p,\mathrm{ac}}\|x\|^p
			\sum_{j=0}^{N-1}(j+1)(a_j-a_{j+1})\\
			&=C_{p,\mathrm{ac}}\|x\|^p\sum_{k=0}^{N-1}a_k\\
			&=C_{p,\mathrm{ac}}\|x\|^p.
		\end{align*}
	\end{proof}
	\begin{proof}[Proof of Proposition \ref{prop:ACB-counterexample-KT} ]
		Let $W$ be the operator from Example~\ref{exSharpACBExponent}
		on $\ell^1(\mathbb N)$ with $\alpha=1/(2p)$.
		By \cite[Theorem~3.4(3) and the proof of Corollary~3.5]
		{AbbarArnoldCoine}, we have $C_{p,\mathrm{ac}}(W)=2$,
		so for every $x\in\ell^1(\mathbb N)$ and every $N\geq1$,
		\begin{equation*}\label{eq:W-ACB}
			\sum_{k=0}^{N-1}\|W^kx\|_1^p\leq2N\|x\|_1^p.
		\end{equation*}
		
		For $0<\delta<1$, put
		\[
		S_\delta=(1-\delta)I+\delta W.
		\]
		Then by Lemma~\ref{lem:ACB-convex-identity}, the operators
		$S_{\delta}$ are (uniformly) $p$-absolutely Ces\`aro bounded
		with $C_{p,\mathrm{ac}}(S_{\delta})
		\leq C_{p,\mathrm{ac}}(W)=2$.
		
		We next obtain a lower bound for the powers of $S_\delta$.
		Let $n\geq1$ satisfy $\mu:=n\delta\geq1$, and put
		\[
		d:=\left\lceil\mu+3\sqrt\mu\right\rceil.
		\]
		Since the vectors $W^ke_d$, $0\leq k<d$, have disjoint supports,
		the binomial formula gives
		\[
		\|S_\delta^ne_d\|_1
		=\sum_{k=0}^{\min(n,d-1)}
		\binom nk\delta^k(1-\delta)^{n-k}
		\left(\frac d{d-k}\right)^{1/(2p)}.
		\]
		
		The identity
		$
		\sum_{k=0}^{n}(k-\mu)^2
		\binom nk\delta^k(1-\delta)^{n-k}
		=n\delta(1-\delta)\leq\mu
		$
		implies that
		\begin{equation}\label{eq:binomial-central-mass}
			\sum_{\substack{0\leq k\leq n\\|k-\mu|\leq2\sqrt\mu}}
			\binom nk\delta^k(1-\delta)^{n-k}
			\geq\frac34.
		\end{equation}
		If $0\leq k\leq n$ and $|k-\mu|\leq2\sqrt\mu$, then,
		since $\mu+3\sqrt\mu-k\leq d-k\leq\mu+3\sqrt\mu-k+1$,
		\begin{equation*}
			\sqrt\mu\leq d-k\leq5\sqrt\mu+1\leq6\sqrt\mu.
		\end{equation*}
		Thus $k<d$ and
		\begin{equation}\label{eq:binomial-distance-bound}
			\left(\frac{d}{d-k}\right)^{1/(2p)}
			\geq\frac{\mu^{1/(4p)}}{6^{1/(2p)}}.
		\end{equation}
		Therefore, by \eqref{eq:binomial-central-mass} and
		\eqref{eq:binomial-distance-bound}, we obtain
		\begin{align}
			\|S_\delta^ne_d\|_1
			&\geq
			\sum_{\substack{0\leq k\leq n\\|k-\mu|\leq2\sqrt\mu}}
			\binom nk\delta^k(1-\delta)^{n-k}
			\left(\frac{d}{d-k}\right)^{1/(2p)}
			\label{eq:Sdelta-lower}\\
			&\geq\frac{\mu^{1/(4p)}}{6^{1/(2p)}}
			\sum_{\substack{0\leq k\leq n\\|k-\mu|\leq2\sqrt\mu}}
			\binom nk\delta^k(1-\delta)^{n-k}
			\notag\\
			&\geq\frac{3}{4\,6^{1/(2p)}}\mu^{1/(4p)}
			=:c_p(n\delta)^{1/(4p)}.
			\notag
		\end{align}
		For $j\geq2$, define
		\[
		\delta_j=\theta_j=\frac1j,
		\qquad
		n_j=\left\lceil j^{8p+1}\right\rceil,
		\qquad
		r_j=\exp(-1/n_j),
		\qquad
		T_j=r_je^{i\theta_j}S_{\delta_j}.
		\]
		On the space $
		X:=\left(\bigoplus_{j\geq 2}\ell^1(\mathbb N)\right)_{\ell^1},
		$
		which is isometrically isomorphic to $\ell^1(\mathbb N)$,
		let
		\[
		T:=\bigoplus_{j\geq2}T_j.
		\]
		The operators $T_j$ are uniformly bounded. Furthermore,
		Minkowski's inequality and $C_{p,\mathrm{ac}}(S_\delta)\leq2$
		imply that, for $x=(x_j)_{j\geq2}\in X$,
		\begin{align}\label{eq:direct-sum-Minkowski}
			\left(\sum_{n=0}^{N-1}\|T^nx\|_X^p\right)^{1/p}
			&=\left(\sum_{n=0}^{N-1}
			\left(\sum_{j\geq2}r_j^n
			\|S_{\delta_j}^nx_j\|_1\right)^p\right)^{1/p}\\
			&\leq\sum_{j\geq2}
			\left(\sum_{n=0}^{N-1}r_j^{np}
			\|S_{\delta_j}^nx_j\|_1^p\right)^{1/p} \notag\\
			&\leq(2N)^{1/p}\sum_{j\geq2}\|x_j\|_1
			=(2N)^{1/p}\|x\|_X. \notag
		\end{align}
		Thus $T$ is $p$-absolutely Ces\`aro bounded with
		$C_{p,\mathrm{ac}}(T)\leq2$.
		
		Since $S_{\delta_j}$ is absolutely Ces\`aro bounded,
		$r(S_{\delta_j})\leq1$, and hence
		\[
		\sigma(T_j)\subseteq r_j\overline{\mathbb D}
		\subset\mathbb D.
		\]
		Moreover,
		\begin{align*}
			\|T_j-I\|
			&=\bigl\|(r_je^{i\theta_j}-1)I
			+r_je^{i\theta_j}\delta_j(W-I)\bigr\|\\
			&\leq |r_je^{i\theta_j}-1|
			+r_j\delta_j\|W-I\|
			\xrightarrow[j\to\infty]{}0.
		\end{align*}
		Fix $\lambda\notin\mathbb D\cup\{1\}$. Each $\lambda I-T_j$
		is invertible, and the preceding convergence implies
		\[
		(\lambda I-T_j)^{-1}
		\xrightarrow[j\to\infty]{}(\lambda-1)^{-1}I
		\]
		in operator norm. Thus these inverses are uniformly bounded,
		and their direct sum defines a bounded inverse of $\lambda I-T$.
		Therefore $\sigma(T)\subseteq\mathbb D\cup\{1\}$.
		
		Finally, choose $\mu_j=(n_j+1)\delta_j$ and
		$d_j=\left\lceil\mu_j+3\sqrt{\mu_j}\right\rceil$.
		Since $S_{\delta_j}$ is positive, we obtain
		\begin{align*}
			\|T_j^{n_j}(I-T_j)e_{d_j}\|_1
			&=r_j^{n_j}
			\left\|S_{\delta_j}^{n_j}e_{d_j}
			-r_je^{i\theta_j}
			S_{\delta_j}^{n_j+1}e_{d_j}\right\|_1\\
			&\geq r_j^{n_j}
			\left\|\operatorname{Im}\left(
			S_{\delta_j}^{n_j}e_{d_j}
			-r_je^{i\theta_j}
			S_{\delta_j}^{n_j+1}e_{d_j}\right)\right\|_1\\
			&=r_j^{n_j+1}\sin(\theta_j)
			\|S_{\delta_j}^{n_j+1}e_{d_j}\|_1\\
			&\geq c_pr_j^{n_j+1}\sin(1/j)
			\left(\frac{n_j+1}{j}\right)^{1/(4p)},
		\end{align*}
		where the last inequality follows from \eqref{eq:Sdelta-lower}.
		Now
		\[
		r_j^{n_j+1}\geq e^{-2},
		\qquad
		\sin(1/j)\geq\frac1{2j},
		\qquad
		\left(\frac{n_j+1}{j}\right)^{1/(4p)}\geq j^2.
		\]
		Consequently,
		\[
		\|T^{n_j}(I-T)\|
		\geq \|T_j^{n_j}(I-T_j)e_{d_j}\|_1
		\geq \frac{c_pe^{-2}}2\,j
		\xrightarrow[j\to\infty]{}\infty.
		\]

	\end{proof}
	
	\begin{remark}
		The construction can be modified to obtain a single operator
		which is $p$-absolutely Ces\`aro bounded for every
		$1\leq p<\infty$, while still satisfying
		\[
		\sigma(T)\subseteq\mathbb D\cup\{1\},
		\qquad
		\limsup_{n\to\infty}\|T^n(I-T)\|=\infty.
		\]
		For $j\geq2$, let $W_j$ be the weighted backward shift on
		$\ell^1(\mathbb N)$ from Example~\ref{exSharpACBExponent}
		with $\alpha=1/j$, and set
		\[
		S_j=(1-1/j)I+j^{-1}W_j,
		\qquad
		n_j=j^{4j+1},
		\qquad
		r_j=e^{-1/n_j},
		\qquad
		T_j=r_je^{i/j}S_j.
		\]
		Define $T=\bigoplus_{j\geq2}T_j$ on
		$X=(\bigoplus_{j\geq2}\ell^1)_{\ell^1}$.
		
		Fix $1\leq p<\infty$. For $j\geq2p$, Example~\ref{exSharpACBExponent} and
		Lemma~\ref{lem:ACB-convex-identity} give
		\[
		C_{p,\mathrm{ac}}(T_j)
		\leq C_{p,\mathrm{ac}}(W_j)
		=\frac1{1-p/j}\leq2.
		\]
		Each of the remaining finitely many blocks is power bounded,
		since
		\[
		\|T_j^n\|\leq r_j^n(n+1)^{1/j}.
		\]
		Hence $\sup_{j\geq2}C_{p,\mathrm{ac}}(T_j)<\infty$.
		As in \eqref{eq:direct-sum-Minkowski}, Minkowski's inequality
		shows that $T$ is $p$-absolutely Ces\`aro bounded (the constants $C_{p,\mathrm{ac}}(T)$ are allowed to depend on $p$). Moreover, as in the proof of Proposition~\ref{prop:ACB-counterexample-KT},
		we have $\sigma(T)\subseteq\mathbb D\cup\{1\}$.
		
		Put $\mu_j=(n_j+1)/j$ and
		$d_j=\lceil\mu_j+3\sqrt{\mu_j}\rceil$.
		The same estimate as in the proof above gives
		\[
		\|S_j^{n_j+1}e_{d_j}\|_1
		\geq\frac{3}{4\,6^{1/j}}\mu_j^{1/(2j)}
		\geq\frac{3}{4\sqrt6}\,j^2,
		\]
		and then
		\[
		\|T^{n_j}(I-T)\|
		\geq r_j^{n_j+1}\sin(1/j)
		\|S_j^{n_j+1}e_{d_j}\|_1
		\geq\frac{3e^{-2}}{8\sqrt6}\,j
		\xrightarrow[j\to\infty]{}\infty.
		\]
		
	\end{remark}

	\section{Concluding remarks}
	
	The results above lead to several further questions concerning
	self-improvement and the asymptotic behavior of powers.
	
	Let $1<p<2$ and let $T$ be an absolutely Ces\`aro bounded operator
	on $L^p(\Omega)$. By Remark \ref{remLpWeakEndpoint}, $T$ is
	$q$-absolutely Ces\`aro bounded for every $1\leq q<p$, while
	Corollary~\ref{corEndpointLogarithm} gives, for every
	$x\in L^p(\Omega)$ and every $N\geq1$,
	\[
	\frac1N\sum_{k=0}^{N-1}\|T^kx\|_p^p
	\leq C\log(N+1)\|x\|_p^p.
	\]
	These estimates provide some progress towards the following
	question, raised by Cuny in \cite{CunyLp}, which remains open.

	\begin{q}
		Let $1<p<2$. Is every absolutely Ces\`aro bounded operator
		on $L^p(\Omega)$ necessarily $p$-absolutely Ces\`aro bounded?
	\end{q}
	
	Equivalently, can the logarithmic factor in the preceding estimate
	be removed? Theorem~\ref{thmEventualPositiveEndpoint} gives an
	affirmative answer under individual eventual positivity.
	The $q$-ACB estimates for $q<p$ do not directly yield
	$p$-absolute Ces\`aro boundedness, since the constants
	$C_{q,\mathrm{ac}}(T)$ need not remain bounded as $q\uparrow p$.
	
	Our example shows that the Katznelson--Tzafriri conclusion need not
	extend from power-bounded operators to $p$-absolutely Ces\`aro bounded
	operators, for every finite $p\geq1$. However, since the operator in
	our construction is not positive, it is natural to ask whether
	positivity ensures the conclusion. This raises the following question.

	\begin{q}
		Let $1<p<\infty$ and let $T$ be a positive absolutely
		Ces\`aro bounded operator on $L^p(\Omega)$. Does the condition $\sigma(T)\subseteq\mathbb D\cup\{1\}$
		imply that
		\[
		\lim_{n\to\infty}\|T^{n+1}-T^n\|=0?
		\]
	\end{q}
	
	In this setting, absolute Ces\`aro boundedness is equivalent
	to $p$-absolute Ces\`aro boundedness. Thus the issue is whether
	positivity, together with these orbit estimates, suffices
	to recover the Katznelson--Tzafriri conclusion.
	
	\section*{Declaration of AI assistance}
	
	During the preparation of this work, the author used GPT-6 Sol
	by OpenAI to assist in the development and refinement of certain
	mathematical arguments. The author thoroughly reviewed and verified
	all AI-assisted content, edited it as needed, and takes full
	responsibility for the content of the article.

\end{document}